\documentclass{amsart}
\usepackage{amssymb}
\usepackage{amsmath}

\newtheorem{theorem}{Theorem}[section]
\newtheorem{lemma}[theorem]{Lemma}

\theoremstyle{definition}

\newcommand{\euO}{\mathfrak O}
\newcommand{\euP}{\mathfrak P}
\newcommand{\euA}{\mathfrak A}
\newcommand{\eub}{\mathfrak b}

\begin{document}
\title[Galois scaffolds]{Galois scaffolds and Galois module
  structure in extensions of characteristic $p$ local
  fields of degree $p^2$}

\author{Nigel P.~Byott and G.~Griffith Elder}
\email{N.P.Byott@exeter.ac.uk}
\email{elder@unomaha.edu}

\address{Mathematics Research Institute, College of Engineering,
  Mathematics and Physical Sciences, University of Exeter, Exeter 
EX4 4QF U.K.}  

\address{Mathematics Dept.,
University of Nebraska at Omaha,
Omaha, NE 68182-0243 U.S.A.}

\date{\today}
\bibliographystyle{amsalpha}

\begin{abstract}
A Galois scaffold, in a Galois extension of local fields with
perfect residue fields, is an adaptation of the normal basis to the
valuation of the extension field, and thus can be applied to answer
questions of Galois module structure.  Here we give a sufficient
condition for a Galois scaffold to exist in fully ramified Galois
extensions of degree $p^2$ of characteristic $p$ local fields.  
This condition becomes necessary when we restrict to $p=3$.  For extensions
$L/K$ of degree $p^2$ that satisfy this condition, we determine the
Galois module structure of the ring of integers by finding necessary and
sufficient conditions for the ring of integers of $L$ to be free over
its associated order in $K[\mbox{Gal}(L/K)]$.
\end{abstract}

%\begin{keyword}
%Galois module structure \MSC 11S15 (11R33)
%\end{keyword}

\maketitle

\section{Introduction}

The Galois module structure of the ring of integers in ramified
$C_p$-extensions of local fields $L/K$ of characteristic $p$ was
studied in \cite{aiba,deSmit}. Of basic importance to that work was a
$K$-basis for the group algebra $K[\mbox{Gal}(L/K)]$ whose effect on
the valuation of the elements of $L$ was easy to determine.  In
\cite{elder:scaffold}, an attempt was made to capture the nice
properties of this basis with the definition of a Galois scaffold.

In this paper, we revise this definition slightly, and show that, in
general, a totally ramified Galois $p$-extension need not admit a
Galois scaffold.  Indeed, the conditions, given in
\cite{elder:scaffold}, that are sufficient for a Galois scaffold to
exist in a fully ramified elementary abelian $p$-extension of
characteristic $p$ local fields are shown here to be necessary for
$C_3\times C_3$-extensions. This is technical work ({\em i.e.}
painstaking linear algebra).  So we take the opportunity here to
extend the results of \cite{elder:scaffold} to
$C_{p^2}$-extensions. Thus in Theorem \ref{scaffold} we give
conditions that are sufficient for a Galois scaffold to exist in any
fully ramified, degree $p^2$ extension of characteristic $p$ local
fields with perfect residue fields, and then prove:

\begin{theorem} \label{assoc-main}
Let $L/K$ be a fully ramified Galois extension of degree $p^2$ that
because it satisfies the conditions of Theorem \ref{scaffold}
possesses a Galois scaffold.  Let $\euA_{L/K}=\{\alpha\in
K[G]:\alpha\euO_L\subseteq \euO_L\}$ be the associated order of the
ring of integers $\euO_L$ of $L$.  Then
$$\euO_L\mbox{ is free over }\euA_{L/K}\mbox{ if and only if }r(b)
\mid p^2-1,$$ where $r(b)$ denotes the least nonnegative residue
modulo $p^2$ of
the second (lower) ramification number of $L/K$.  Furthermore, if
$\euO_L$ is free over $\euA_{L/K}$ then any element $\rho\in L$
with normalized valuation $v_L(\rho)=r(b)$ satisfies
$\euO_L=\euA_{L/K}\rho$.
\end{theorem}
The proof of this result appears in \S2.4.

\subsection{Notation}
Let $p$ be prime and let $\mathbb{F}_p$ be the finite field with $p$ elements.
Let
$\kappa$ be a perfect field containing $\mathbb{F}_p$, let
$K_0=\kappa((t))$ be the local function field with residue field
$\kappa$, and let $K_n/K_0$ be a fully ramified Galois extension of
degree $p^n$ with Galois group $G=\mbox{Gal}(K_n/K_0)$.  The
ramification filtration of $G$ is the set of subgroups
$G_i=\{\sigma\in G: v_n((\sigma-1)\pi_n)\geq i+1\}$.  Subscripts denote
field of reference. So, for example, $v_n$ is the additive valuation on $K_n$,
normalized so that $v_n(K_n^\times)=\mathbb{Z}$, $\pi_n$ is a prime
element of $K_n$ with $v_n(\pi_n)=1$, and $\euO_n=\{x\in K_n:
v_n(x)\geq 0\}$ is the valuation ring with maximal ideal $\euP_n=
\{x\in K_n: v_n(x)> 0\}$.

Quotients of consecutive ramification groups $G_i/G_{i+1}$ are either
trivial or elementary abelian $C_p\times \cdots \times C_p$
\cite[IV\S2 Prop 7 Cor 3]{serre:local}. Thus the usual ramification
filtration can be refined: There is a filtration $G=H_0\supsetneq
H_1\supsetneq \cdots \supsetneq H_{n-1}\supsetneq H_n=\{1\}$ such that
$H_i/H_{i+1}\cong C_p$ for $0\leq i\leq n-1$ and $\{H_i:0\leq i\leq
n\}\supseteq \{G_i:i\geq 1\}$. Choose one such filtration. Choose elements
$\sigma_{i+1}\in H_i\setminus H_{i+1}$ for each $0\leq i\leq n-1$ and
define $b_i=v_n((\sigma_{i}-1)\pi_n)-1$. Then $b_1\leq b_2\leq \cdots
\leq b_n$. Define the ramification multiset to be $\{b_i:1\leq i\leq
n\}$, which is independent of our choices \cite[IV\S1 Prop 3
  Cor]{serre:local}, and thus should be considered a fundamental
invariant of the extension.  As a set, it is just the set of (lower)
ramification numbers, subscripts $i$ with $G_i\supsetneq G_{i+1}$.
 
Define $K_i=K_n^{H_i}$ to be the fixed field of $H_i$. Thus we have
a path through the subfields of $K_n$, from $K_n$ down
to $K_0$, which is consistent with the ramification multiset:
$\{b_i:j< i\leq n\}$ is the ramification multiset for
$K_n/K_j$, $\{b_i:0< i\leq j\}$ is the ramification multiset
for $K_j/K_0$, and $b_i$ is the ramification number for
$K_i/K_{i-1}$.

Let $\euA_{K_n/K_0}=\{\alpha\in K_0[G]:\alpha\euO_n\subseteq \euO_n\}$
denote the associated order of $\euO_n$ in the group algebra $K_0[G]$.
Since $\euA_{K_n/K_0}$ is an $\euO_0$-order in $K_0[G]$ containing
$\euO_0[G]$ and $\euO_n$ is a module over $\euA_{K_n/K_0}$, it is
natural to ask about the structure of $\euO_n$ over $\euA_{K_n/K_0}$.
Although more general questions can be addressed ({\em
  e.g.}~\cite{deSmit}), we follow \cite{aiba, byott:scaffold} here and
focus our attention on determining conditions that are necessary and
sufficient for $\euO_n$ to free over $\euA_{K_n/K_0}$.

Let $\lfloor x\rfloor$ and $\lceil x\rceil$ denote the greatest
integer and least integer functions, respectively.  Let
$\wp(X)=X^p-X\in \mathbb{Z}[X]$ and $\binom{X}{i}=X(X-1)\cdots
(X-i+1)/i!$ denote the binomial coefficient. Define truncated
exponentiation by the following truncation of the binomial series:
$$(1+X)^{[Y]}:=\sum_{i=0}^{p-1}\binom{Y}{i}X^i\in\mathbb{Z}_{(p)}[X,Y],$$
where $\mathbb{Z}_{(p)}$ is the integers localized at $p$.  
Vandermonde's Convolution Identity is
$\sum_{i=0}^{t}\binom{Y}{i}\binom{X}{t-i}
=\binom{X+Y}{t}\in\mathbb{Z}_{(p)}[X,Y]$
for $0\leq t\leq p-1$.

\subsection{Definition of Galois scaffold}
The term was introduced in \cite{elder:scaffold}. Its definition is
refined here.  Two ingredients are required: A valuation criterion for
a normal basis generator and a generating set for a particularly nice
$K_0$-basis of the group algebra $K_0[G]$.

\subsubsection{Valuation criterion}
In a Galois extension of local fields $K_n/K_0$, a valuation criterion
for a normal basis generator is an integer $c$ such that if $\rho\in
L$ with $v_n(\rho)=c$ then $\{\sigma\rho:\sigma\in G\}$ is a normal
basis for $K_n$ over $K_0$.  For fields of characteristic $p$, every
totally ramified Galois $p$-extension $K_n/K_0$ has a valuation
criterion. Indeed, if the extension is abelian, $c$ can then be any
integer $c\equiv b_n\bmod p^n$ \cite{elder:cor-criterion}.

\subsubsection{Generating set for the group algebra $K_0[G]$}
We have chosen a refined filtration $\{H_i\}$ of the Galois group
along with group elements $\sigma_i\in H_{i-1}\setminus H_i$. These
elements certainly generate the Galois group,
$G=\{\prod_{i=1}^n\sigma_i^{a_i}:0\leq a_i\leq p-1\}$, and thus generate
a basis for $K_0[G]$ over $K_0$, a basis that is naturally
associated with a normal basis for $K_n/K_0$.  A Galois
scaffold occurs if there is a similar generating set of $n$ elements
$\{\Psi_i\}$ from the augmentation ideal $(\sigma-1:\sigma\in G)$ of
$K_0[G]$ that satisfies
a regularity condition and a spanning condition:
For all $0\leq j<p$ and all
$\rho,\rho'\in K_n$ that satisfy the valuation criterion, 
$v_n(\rho),v_n(\rho')\equiv c\bmod p^n$, 
\begin{equation}\label{strong-scaff}
v_n(\Psi_i^j\rho)-v_n(\rho)=j\cdot( (v_n(\Psi_i\rho')-v_n(\rho')).
\end{equation}
For $0\leq a<p^n$, define $\Psi^{(a)} = \Psi_{n}^{a_{(0)}}\Psi_{n-1}^{a_{(1)}}\cdots \Psi_{1}^{a_{(n-1)}}$
where $a$ is expanded $p$-adically as $a=\sum_i
a_{(i)} p^i$ with $0\leq a_{(i)}<p$.  Then
for $v_n(\rho)\equiv c\bmod p^n$,
\begin{equation}\label{scaff}
\left \{v_n(\Psi^{(a)}\rho):0\leq a<p^n\right \}
\end{equation} is a complete set
of residues modulo $p^n$.  Because $K_n/K_0$ is fully
ramified of degree $p^n$, this means that $\{\Psi^{(a)}:0\leq a<p^n\}$
is a $K_0$-basis for $K_0[G]$.

A quick comment now about the definition of Galois
scaffold in \cite{elder:scaffold}. While we explicitly require a
Galois scaffold here to have two properties, (\ref{strong-scaff}) and
(\ref{scaff}), the definition stated in \cite{elder:scaffold} required only
(\ref{scaff}) explicitly.  Note however that the Galois scaffold given in
\cite{elder:scaffold} did satisfy both (\ref{strong-scaff}) and
(\ref{scaff}).

\section{Galois extensions of degree $p^2$ with Galois scaffold and their resulting Galois module structure}
\subsection{Characterizing the extensions}

Elementary abelian extensions of degree $p^2$ correspond to
2-dimensional subspaces of $K_0 / \wp(K_0)$, where $\wp(K_0)=\{ \wp(k)
: k \in K_0\}$. Cyclic extensions of degree $p^2$ correspond to Witt
vectors $(\beta_1,\beta_2)$ of length 2, and the extension is
unchanged if we add an element of $\wp(K_0)$ to $\beta_1$ or
$\beta_2$. Thus, in either case, the extensions are determined by a
pair of coset representations of $\wp(K_0)$. In this subsection, we
explain these correspondences and tie those coset representatives
(reduced representatives) that are distinguished for having maximal
valuation to   the ramification numbers for $K_2/K_0$. We also set
up notation for the Galois action that is consistent with \S1.1.
\subsubsection{Elementary abelian}
The map that takes $K_2=K_0(x_1,x_2)$ with $\wp(x_i)=\beta_i\in K_0$
to $V=\mathbb{F}_p\beta_1+\mathbb{F}_p\beta_2+\wp(K_0)$ sets up
bijection between $C_p\times C_p$-extensions of $K_0$ and
$2$-dimensional $\mathbb{F}_p$-vector spaces of $K_0/\wp(K_0)$.  Given
such a subspace $V$, choose $\beta_1$ so that
$v_0(\beta_1)=\max\{v_0(\beta):\beta\in V\}$.  Choose $\beta_2 \in V$
so that $\beta_1$ and $\beta_2$ span $V$ and replace $\beta_2$ by
another representative of $\beta_2+\mathbb{F}_p\beta_1+\wp(K_0)$ if
necessary so that $v_0(\beta_2)=\max\{v_0(\beta):\beta\in
\beta_2+\mathbb{F}_p\beta_1+\wp(K_0)\}$.  As a result,
$v_0(\beta_i)=-u_i$ with $0\leq u_1\leq u_2$ and $p\nmid u_i$ unless
$u_1=0$, in which case $K_2/K_0$ is not fully ramified.

Restrict to the situation where $K_2/K_0$ is fully ramified. Then 
because of our choices for $\beta_1$ and $\beta_2$,
$\{u_1, u_2\}$ is the set of upper ramification numbers for $K_2/K_0$.
The lower ramification numbers are $b_1=u_1$ and $b_2=u_1+p(u_2-u_1)$
\cite[IV \S3]{serre:local}.  Choose $\sigma_i\in G$ so that
$(\sigma_i-1)x_j=\delta_{i,j}$ where
$$\delta_{i,j}=\begin{cases}1&\mbox{for }i=j,\\
0&\mbox{for }i\neq j.\end{cases}$$
Set $H_1=\langle\sigma_2\rangle$, so that $K_0(x_1)=K_1=K_2^{\sigma_2}$.
Since the norm
$N_{K_1/K_0}(x_1)=\wp(x_1)=\beta_1$, we have $v_1(x_1)=-b_1$ as well.
Similarly, $v_2(x_2)=-pu_2$.

\subsubsection{Cyclic}
As shown in \cite{schmid:1936,schmid:1937}, each $C_{p^2}$-extension
of $K_0$ can be associated with a Witt vector $(\beta_1, \beta_2)$. We
can assume that $\beta_1\in K_0$ is the element of maximum valuation
in its nonzero coset of $\wp(K_0)$ and that $\beta_2\in K_0$ is a
element of maximum valuation in its coset of
$\mathbb{F}_p\beta_1+\wp(K_0)$.  If we abuse notation by identifying
these cosets with their representatives, this gives a bijection
between $C_{p^2}$-extensions and the one-dimensional
$\mathbb{F}_p$-vector spaces
$\{(a\beta_1,a\beta_2):a\in\mathbb{F}_p\}$.

Restrict now to the situation where $K_2/K_0$ is fully ramified. Let
$\sigma_1$ generate the Galois group $G$, and set
$\sigma_2=\sigma_1^p$ and $H_1=\langle \sigma_2\rangle$.  Then
$K_1=K_0(x_1)$, where $\wp(x_1)=\beta_1$, is the fixed field of
$\sigma_2$.  Without loss of generality, $(\sigma_1-1)x_1=1$. Our
choice of $\beta_1$ means that $v_0(\beta_1)=-b_1<0$ with $p\nmid
b_1$.  Since $\beta_1$ is the norm of $x_1$, $v_1(x_1)=-b_1$. Thus
$b_1$ is the ramification number for $K_1/K_0$, and also the first
(lower) ramification number for $K_2/K_0$.

The second (lower) ramification number $b_2$ of $K_2/K_0$ is also the
ramification number for $K_2/K_1$. It is dependent upon both
$v_0(\beta_1)=-b_1$ and $v_0(\beta_2)=-u_2^*$, which due to our
assumption on $\beta_2$ satisfies $0\leq u_2^*$ and if $u_2^*\neq 0$
then $p\nmid u_2^*$.  Indeed, we will proceed now to show that
$b_2=\max\{(p^2-p+1)b_1,pu_2^*-(p-1)b_1\}$, and thus that the upper
ramification numbers are $u_1=b_1<u_2=\max\{pb_1,u_2^*\}$.

Let
$D_1=(x_1^p+\beta_1^p-(x_1+\beta_1)^p)/p=-\sum_{i=1}^{p-1}\frac{1}{p}\binom{p}{i}x_1^i\beta_1^{p-i}\in
K_1$. Observe that $v_1(D_1)=-(p^2-p+1)b_1$.  As explained in
\cite{schmid:1936,schmid:1937}, $K_0(x^*_2)$ with $\wp(x^*_2)=D_1$ is
a $C_{p^2}$-extension of $K_0$ that contains $K_1$ (and is associated
with the Witt vector $(\beta_1,0)$). 
Moreover, every
$C_{p^2}$-extension of $K_0$ that contains $K_1$ arises as
$K_2=K_0(x_2)$ with $\wp(x_2)=D_1+\beta_2$.
Then $x_2=x_2^*+z_2$ 
where $\wp(z_2)=\beta_2$,
and
$K_2=K_0(x_2)$ is contained in the $C_{p^2}\times C_p$-extension
$K_0(x_2^*,z_2)$.  Without loss of generality, we may assume that
$\sigma_1\in\mbox{Gal}(K_0(x_2^*,z_2)/K_0)$ satisfies $(\sigma_1-1)
z_2=0$. Furthermore
$(\sigma_1-1)x_2=(\sigma_1-1)x_2^*=C_1$ where $C_1=(x_1^p+1-(x_1+1)^p)/p=
-\sum_{i=1}^{p-1}\frac{1}{p}\binom{p}{i}x_1^i$, and $(\sigma_2-1)
x_2=1$.  Notice that $v_1(C_1)=-(p-1)b_1$.

We now work with the ramification filtrations of two different
$C_p\times C_p$-extensions: $K_0(x_1,z_2)/K_0$ and $K_1(x_2^*,
z_2)/K_1$.  There are three possibilities for the set of upper
ramification numbers for $K_0(x_1,z_2)/K_0$: If $b_1\neq u_2^*$, the
set is $\{b_1,u_2^*\}$.  If $b_1=u_2^*$, the set is either $\{b_1\}$
or $\{b_1,v\}$ (for some $v<b_1$). In each case, we pass to the lower
ramification numbers for $K_0(x_1,z_2)/K_0$, using \cite[IV
  \S3]{serre:local}. The ramification number for $K_1(z_2)/K_1$ is
therefore $b_1+p(u_2^*-b_1)$ (when $u_2^*>b_1$) or some integer $\leq
b_1$ (when $u_2^*\leq b_1$).  Now consider $K_1(x_2^*, z_2)/K_1$.  It
is easy to see that the ramification number for $K_1(x_2^*)/K_1$ is
$-v_1(D_1)=(p^2-p+1)b_1$. This means, since if $u_2^*\neq 0$ then
$p\nmid u_2^*$, that the ramification numbers for $K_1(x_2^*)/K_1$ and
for $K_1(z_2)/K_1$ are distinct. As a result, these are the two
distinct upper ramification numbers for $K_1(x_2^*, z_2)/K_1$.
Passing to the lower ramification numbers for $K_1(x_2^*, z_2)/K_1$,
considering all the cases, we find that the ramification number of
$K_1(x_2)/K_1$ is $b_2=\max\{(p^2-p+1)b_1,pu_2^*-(p-1)b_1\}$.

\subsection{The Galois scaffold}
Since $p\nmid v_0(\beta_1)$,
the set
$\{v_0(\beta_1^t):0\leq t\leq p-1\}$ is a complete set of residues
modulo $p$. As a result, it is generically the case that $\beta_2=\sum_{t=0}^{p-1}\mu_t^p\beta_1^t$ for
some $\mu_t\in K_0$. Moreover, since we are only interested in the
expression for $\beta_2$ in $K_0/K_0^\wp$, we may assume that the
$t=0$ term satisfies $\mu_0^p\in \kappa$.  Gather all terms except
$\mu_1^p\beta_1$ into an ``error term'' $\epsilon$. Replace $\mu_1$
with $\mu$, and let $m=-v_0(\mu)$.
Thus
$$\beta_2=\mu^p\beta_1+\epsilon$$ where we may assume either
$\epsilon\in \kappa$ or $p\nmid v_0(\epsilon)=-e<0$.  Note that
$v_0(\epsilon)\not\equiv v_0(\mu^p\beta_1)\bmod p$. Thus $e\not\equiv
b_1\bmod p$.  We are now prepared to state:

\begin{theorem}\label{scaffold}
Let $K_2/K_0$ be a fully ramified Galois extension of degree $p^2$.
Adopt the notation of this section, and assume that
$v_0(\epsilon)>v_0(\beta_2)+(p-1)b_1/p$. For $G\cong C_{p^2}$,
additionally assume $v_0(\beta_1^p)>v_0(\beta_2)+(p-1)b_1/p$. Then
there is a Galois scaffold.  Define $\Psi_1\in K_0[G]$ by
$$\Psi_1+1=
\sigma_1\sigma_2^{[\mu]}=\sigma_1\sum_{i=0}^{p-1}\binom{\mu}{i}(\sigma_2-1)^i.$$
Let $\Psi_2=\sigma_2-1$.
Then for $\alpha\in K_2$ with $v_2(\alpha)\equiv b_2\bmod p^2$ and
$0\leq i,j\leq p-1$,
$$v_2\left(\Psi_2^i\Psi_1^j\alpha\right)=v_2(\alpha)+ib_2+jpb_1.$$
\end{theorem}
The proof of this theorem appears in \S2.3. First, we examine its
conditions in terms of the ramification numbers for $K_2/K_0$.  In
\S2.1.2, we observed that for $G\cong C_{p^2}$, 
$b_2=\max\{(p^2-p+1)b_1,pu_2^*-(p-1)b_1\}$.  The requirement that
$v_0(\beta_1^p)>v_0(\beta_2)+(p-1)b_1/p$ means that
$pu_2^*-(p-1)b_1>p^2b_1$. Thus for $G\cong C_{p^2}$, $u_2=u_2^*$,
$b_2=pu_2-(p-1)b_1$ and so the requirement that
$v_0(\beta_1^p)>v_0(\beta_2)+(p-1)b_1/p$ is a strengthening of the
lower bound on $b_2$, from $b_2\geq (p^2-p+1)b_1$ to
\begin{equation}
b_2>p^2b_1.
\label{cyc-spread}
\end{equation}
The other condition $v_0(\epsilon)>v_0(\beta_2)+(p-1)b_1/p$, which is
a restriction for both $G\cong C_p\times C_p$ and $C_{p^2}$, implies
$\beta_2\equiv \mu^p\beta_1\bmod \mu^p\beta_1\euP_2$. Using 
$b_2=b_1+
p(u_2-b_1)$ (and thus \eqref{cyc-spread}
when $G\cong C_{p^2}$), this means that
$v_0(\epsilon)>v_0(\beta_2)+(p-1)b_1/p$ 
can be rewritten as
\begin{equation}
b_2>pe.
\label{cyc-shape}
\end{equation}

\subsection{Proof of Theorem \ref{scaffold}}
The result for $G\cong C_p\times C_p$ follows from \cite[Thm
  4.1]{elder:scaffold}. So we focus here on the result for $G\cong
C_{p^2}$ and recall the notation of \S2.1.2. There are three steps in
our argument.  Thus three subsections.
\subsubsection{An explicit element satisfying the valuation criterion}
The hypothesis on $v_0(\epsilon)$ ensures at least that
$v_0(\epsilon) > v_0(\beta_2)=v_0(\mu^p \beta_1)$, so that $-b_1-pm <
-e$ and $u_2^*=pm + b_1$. Under this weaker assumption, we determine
$\epsilon_1\in K_1$ such that $X_2=x_2-\mu x_1+\epsilon_1\in K_2$ has
valuation $v_2(X_2)=-b_2=-\max\{pu_2^*-(p-1)b_1,(p^2-p+1)b_1\}$.  Once
this is done, $\rho=\binom{X_2}{p-1}\binom{x_1}{p-1}\in K_2$ satisfies
$v_2(\rho)\equiv b_2\bmod p^2$.

The element $\epsilon_1\in K_1$ is determined by $\epsilon$.
Recall that either $\epsilon\in \kappa$ or $p\nmid -e<0$.  If
$\epsilon\in \kappa$, we simply let $\epsilon_1=0$ (and also set
$E_1=\epsilon$).  The interesting case occurs when $\epsilon\not\in \kappa$
and thus $K_1(z)/K_0$ with $\wp(z)=\epsilon$ is a fully ramified
$C_p\times C_p$ extension with upper ramification numbers
$e=-v_0(\epsilon)$ and $b_1$. Recall $e\not\equiv b_1\bmod p$. So
$e\neq b_1$.  Passing to the lower numbering for $K_1(z)/K_0$ using
\cite[IV \S3]{serre:local}, we find that the ramification number for
$K_1(z)/K_1$ is $\max\{e,b_1+p(e-b_1)\}$ (either $e$ when $e< b_1$, or
$pe-(p-1)b_1$ when $e>b_1$). Using this information regarding
$K_1(z)/K_1$ there must be a coset representative $E_1$ for the coset
$\epsilon+\wp(K_1)$ in $K_1/\wp(K_1)$ such that
$v_1(E_1)=-\max\{e,b_1+p(e-b_1)\}$.  Thus
$E_1=\epsilon+\wp(\epsilon_1)$ for some
$\epsilon_1\in K_1$.  Since $v_1(E_1)>v_1(\epsilon)$, we have
$-pe=v_1(\epsilon)=v_1(\wp(\epsilon_1))$.  This means that
$v_1(\epsilon_1)=-e$.

Observe, based upon \S2.1.2, that
$\wp(x_2)=D_1+\beta_2=D_1+\mu^p\beta_1+\epsilon$ and $\wp(\mu
x_1)=\mu^px_1^p-\mu x_1=\mu^p(x_1+\beta_1)-\mu
x_1=\wp(\mu)x_1+\mu^p\beta_1$. Therefore $\wp(X_2)=D_1-\wp(\mu
)x_1+E_1\in K_1$. 
Because $-b_1-pm<-e$, $v_1(\wp(\mu
)x_1)=-b_1-p^2m<(p-1)b_1-pe\leq v_1(E_1)$. Thus 
$v_1(-\wp(\mu)x_1+E_1)=-b_1-p^2m$. Furthermore
$v_1(D_1)=-(p^2-p+1)b_1$. Thus
$v_1(\wp(X_2))=\min\{-b_1-p^2m,-(p^2-p+1)b_1\}=-b_2$. Since
$\mbox{Norm}_{K_2/K_1}(X_2)=\wp(X_2)$, $v_2(X_2)=-b_2$.

\subsubsection{A Galois scaffold for the explicit element in \S2.3.1}
Observe that
$(\sigma_1-1)X_2=C_1-\mu+(\sigma-1)\epsilon_1$
and thus
$(\sigma_1-1)X_2=-\mu+\mathcal{E}$ where $\mathcal{E} =C_1+(\sigma_1-1)\epsilon_1\in
K_1$ satisfies $v_1(\mathcal{E} )=\min\{-(p-1)b_1,b_1-e\}$. Note that for
$e>0$ we have $p\nmid e$. So $(p-1)b_1\neq e-b_1$. In any case,
(\ref{cyc-spread}) means $-(p-1)b_1>b_1-b_2/p$ 
and (\ref{cyc-shape}) means that $b_1-e>b_1-b_2/p$. Together they yield
$v_1(\mathcal{E} )>b_1-b_2/p$. Thus $v_2(\mathcal{E} )>pb_1-b_2$.

Using truncated
exponentiation and Vandermonde's Convolution Identity, 
$$\sigma_2^{[\mu]}\binom{X_2}{p-1}=\sum_{i=0}^{p-1}\binom{\mu}{i}(\sigma_2-1)^i
\binom{X_2}{p-1}=\sum_{i=0}^{p-1}\binom{\mu}{i} \binom{X_2}{p-i-1}=
\binom{X_2+\mu}{p-1}.$$ Therefore
$\sigma_1\sigma_2^{[\mu]}\binom{X_2}{p-1}=\binom{X_2+\mathcal{E} }{p-1}$.  If
we expand $\binom{X_2+\mathcal{E} }{p-1}$ 
using Vandermonde's Convolution Identity, we
notice that for $0\leq i<p-1$,
$v_2(\binom{X_2}{i}\binom{\mathcal{E}}{p-i-1})>(p-i-1)pb_1-(p-1)b_2$.
So $v_2(\binom{X_2}{i}\binom{\mathcal{E}}{p-i-1})> pb_1-(p-1)b_2=v_2(\binom{X_2}{p-1}/x_1)$
for $0\leq i<p-1$ and thus
$$\sigma_1\sigma_2^{[\mu]}\binom{X_2}{p-1}=\binom{X_2+\mathcal{E} }{p-1}\equiv
\binom{X_2}{p-1}\bmod \binom{X_2}{p-1}\frac{1}{x_1} \euP_2.$$

Let $\Psi_1=\sigma_1\sigma_2^{[\mu]}-1$ and observe that for
$0\leq i\leq p-1$, $\Psi_1^i\binom{X_2}{p-1}\binom{x_1}{p-1}\equiv
\binom{X_2}{p-1}\binom{x_1}{p-i-1}\bmod
\binom{X_2}{p-1}\binom{x_1}{p-i-1}\euP_2$, which means that
with $\Psi_2=\sigma_2-1$,
$$\Psi_2^i\Psi_1^j\rho\equiv
\binom{X_2}{p-i-1}\binom{x_1}{p-j-1}\bmod
\binom{X_2}{p-i-1}\binom{x_1}{p-j-1}\euP_2,$$ and therefore
$v_2(\Psi_2^i\Psi_1^j\rho)=v_2(\rho)+ib_2+jpb_1$ for
$0\leq i,j\leq p-1$.
Note that $\{v_2(\rho)+ib_2+jpb_1:
0\leq i,j\leq p-1\}$
is a complete set of residues modulo $p^2$.

\subsubsection{The Galois scaffold holds for any element 
$\alpha\in K_2$ with $v_2(\alpha)\equiv b_2\bmod p^2$} Express
$\alpha\in K_2$ with $v_2(\alpha)\equiv b_2\bmod p^2$ in the
$K_0$-basis $\{\Psi_2^m\Psi_1^n\rho: 0\leq m,n\leq p-1\}$.  So
$\alpha=\sum_{0\leq m,n<p}A_{m,n}\Psi_2^m\Psi_1^n\rho$ for
some $A_{i,j}\in K_0$.  Since $v_2(\alpha)\equiv v_2(\rho)\bmod p^2$,
$A_{0,0}\neq 0$ and it will be enough to prove the result
for $\alpha'=\alpha/A_{0,0}$. Therefore, without loss of
generality, we assume $A_{0,0}=1$ and $v_2(A_{m,n})+mb_2+npb_1>0$ for
$(m,n)\neq (0,0)$. Now apply $\Psi_2^i\Psi_1^j$ for $0\leq
i,j\leq p-1$ to $\alpha$.  Clearly
$v_2(\Psi_2^i\Psi_1^j\rho)=v_2(\alpha)+ib_2+jpb_1$. The only
question then is whether $v_2(\Psi_2^i\Psi_1^j\cdot A_{m,n}
\Psi_2^m\Psi_1^n\rho)>v_2(\alpha)+ib_2+jpb_1$ for $(m,n)\neq
(0,0)$. Since $\Psi_2^p=0$ and $\Psi_1^p=\Psi_2$, the
interesting cases, when $\Psi_2^i\Psi_1^j\cdot
\Psi_2^m\Psi_1^n\neq 0$, occur only when $j+n<p$ and $i+m<p$,
or $j+n\geq p$ and $i+m+1<p$. Consider them separately. If $j+n<p$ and
$i+m<p$, then $v_2(\Psi_2^i\Psi_1^j\cdot A_{m,n}
\Psi_2^m\Psi_1^n\rho)=v_2(\rho)+v_2(A_{m,n})+(i+m)b_2+(j+n)pb_1>
v_2(\rho)+ib_2+jpb_1$. Of course $v_2(\rho)=v_2(\alpha)$.  If $j+n\geq
p$ and $i+m+1<p$, then $v_2(\Psi_2^i\Psi_1^j\cdot A_{m,n}
\Psi_2^m\Psi_1^n\rho)=v_2(\rho)+v_2(A_{m,n})+(i+m+1)b_2+(j+n-p)pb_1>
v_2(\rho)+ib_2+jpb_1+(b_2-p^2b_1)$. Recall restriction
(\ref{cyc-spread}) that $b_2>p^2b_1$.

\subsection{Proof of Theorem \ref{assoc-main}}
The proof for $G=\mbox{Gal}(K_2/K_0)\cong C_p\times C_p$ is contained
in \cite{byott:scaffold}. Here we adjust that argument so that it
applies to $G\cong C_{p^2}$.  Let $K_2/K_0$ satisfy the conditions in
Theorem \ref{scaffold}. So, in particular, $b_2\equiv b_1\equiv
r(b)\bmod p^2$.  Recall
$\Psi_1^p=\Psi_2$ and
$\Psi_2^p=0$. This means that if we represent every nonnegative
integer $p$-adically ({\em i.e.} for $a\in\mathbb{Z}$ with $a\geq 0$
write $a=\sum_{i=0}^{\infty}a_{(i)}p^i$ for some $0\leq a_{(i)}\leq
p-1$), then we may define
$$\Psi^{(a)}=\begin{cases}\Psi_2^{a_{(1)}}\Psi_1^{a_{(0)}}& a<p^2,\\
0 &\mbox{ otherwise,} \end{cases}$$
and find that
$\Psi^{(a)}\Psi^{(a')}=\Psi^{(a+a')}$.
Furthermore,
if we define a function $\eub$ from the nonnegative integers to $\mathbb{Z}\cup\{\infty\}$:
$$\eub(a)=\begin{cases} (1+a_{(1)})b_2+a_{(0)}pb_1 & a<p^2,\\ \infty
&\mbox{ otherwise,} \end{cases}$$ then because of
Theorem \ref{scaffold}, given any $\rho\in K_2$ with
$v_2(\rho)= b_2$, we have
$v_2(\Psi^{(a)}\rho)=\eub(a)$.
For $0\leq a<p^2$, set
$$d_a=\left\lfloor\frac{\eub(a)}{p^2}\right\rfloor.$$
So $\eub(a)=d_ap^2+r(\eub(a))$ where $r(\eub(a))$
is the least nonnegative residue modulo $p^2$.

Let
$\rho_*\in K_2$ with $v_2(\rho_*)=r(b_2)$. Recall that $t$ is a
uniformizer for $K_0=\mathbb{F}((t))$. Set $\rho=t^{d_0}\rho_*$, so
$v_2(\rho)=b_2$. Moreover, for $0\leq a$ set
$$\rho_a=t^{-d_a}\Psi^{(a)}\cdot \rho,$$ which means that $\rho_a=0$
for $a\geq p^2$. Note that $v_2(\rho_a)=r(\eub(a))$ for $0\leq
a<p^2$. Thus $\{v_2(\rho_a):0\leq a<p^2\}=\{0,\ldots, p^2-1\}$, 
$\{\rho_a\}_{0\leq a<p^2}$ is an $\euO_0$-basis for $\euO_2$, and 
the elements $\Psi^{(a)}\rho$
span $K_2$ over $K_0$. By comparing dimensions, we see that $\rho$
generates a normal basis for the extension $K_2/K_0$, and
$\{\Psi^{(a)}\}_{0\leq a<p^2}$ is a $K_0$-basis for the group algebra
$K_0[G]$.
Observe that
\begin{equation}\label{psi-action}
\Psi^{(a_1)}\cdot \rho_{a_2}=t^{d_{a_1+a_2}-d_{a_2}}\rho_{a_1+a_2},
\end{equation}
and define
$w_j=\min\{d_{j+a}-d_{a}:0\leq a\leq j+a<p^2\}$ where
$0\leq j<p^2$. Note, in particular, that
$w_0=0$ and that we have $w_j\leq d_j-d_0$ for all $j$.

\begin{lemma}\label{d-lemma} 
The associated order $\euA_{K_2/K_0}$ of $\euO_2$ has $\euO_0$-basis $\{t^{-w_j}\Psi^{(j)}\}_{0\leq j<p^2}$.
Moreover, $\euO_2$ is a free module over $\euA_{K_2/K_0}$ if and only if
$w_j=d_j-d_0$ for all $j$, and in this case $\rho_*$ is a
free generator of $\euO_2$ over $\euA_{K_2/K_0}$.
\end{lemma}

\begin{proof}
Follow \cite[Theorem 2.3]{byott:scaffold}.  Since
$\{\Psi^{(j)}:0\leq j<p^2\}$ is a $K_0$-basis for $K_0[G]$, any
element $\alpha\in K_0[G]$ may be written
$\alpha=\sum_{j=0}^{p^2-1}c_j\Psi^{(j)}\mbox{ with }c_j\in K_0$.
Using (\ref{psi-action}) and the fact that $\{\rho_a\}_{0\leq a<p^2}$
is an $\euO_0$-basis for $\euO_2$, we find that $\alpha\in
\euA_{K_2/K_0}$ is equivalent to
$\alpha\rho_a=\sum_{j=0}^{p^2-1}c_j\Psi^{(j)}\rho_a\in \euO_2$ for all
$0\leq a<p^2$.  This in turn is equivalent to $c_jt^{d_{j+a}-d_a}\in
\euO_0$ or $v_0(c_j)\geq d_a-d_{j+a}$ for all $0\leq a\leq
a+j<p^2$. But this is equivalent to $-v_0(c_j)\leq w_j$ for all $0\leq
j<p^2$. The first statement is proven.

Consider the second.  Suppose that $w_j=d_j-d_0$ for all $j$. As
$\rho_*=\rho_0$, (\ref{psi-action}) yields $t^{-w_j}\Psi^{(j)}\cdot
\rho_*=\rho_a$, the basis elements $\{t^{-w_j}\Psi^{(j)}:0\leq
j<p^2\}$ take $\rho_*$ to the basis elements $\{\rho_j:0\leq j<p^2\}$
of $\euO_2$, which means that $\euO_2$ is a free
$\euA_{K_2/K_0}$-module. Conversely, suppose that $\euO_2$ is a free
$\euA_{K_2/K_0}$-module. So $\euO_2=\euA_{K_2/K_0}\eta$ for some
$\eta\in K_2$. Since $1\in \euA_{K_2/K_0}$, $\eta\in \euO_2$ and so
$\eta=\sum_{r=0}^{p^2-1}x_r\rho_r$ for some $x_r\in \euO_0$.  We have
two $\euO_0$-bases for $\euO_2$, $\{\rho_j:0\leq j<p^2\}$ and
$\{t^{-w_i}\Psi^{(i)}\eta :0\leq i<p^2\}$. Because of
(\ref{psi-action}) the matrix that takes 
the first of these to the second,
namely $M=(a_{i,j})$, is
upper triangular with
$$a_{i,j}=\begin{cases}0&i>j,\\
x_{j-i}t^{d_j-d_{j-i}-w_i} &i\leq j.
\end{cases}$$
Furthermore, it must have coefficients in $\euO_0$ and unit
determinant. Recall $x_r\in\euO_0$, so in particular $x_0\in\euO_0$.
Because the coefficients on the diagonal lie in $\euO_0$,
$x_0t^{-w_j+d_j-d_0}\in\euO_0$. Because the determinant 
$\prod_{j=0}^{p^2-1}a_{j,j}=x_0^{p^2}
\prod_{j=0}^{p^2-1}t^{d_j-d_0-w_j}$
is a unit,
we have $w_j=d_j-d_0$ for all
$0\leq j<p^2$, as required.
\end{proof}

The condition $w_j=d_j-d_0$ for all $0\leq j<p^2$ can be restated as
$d_{x + y}-d_x\geq d_y-d_0$ for all $0\leq y<p^2$ and $0\leq
x<p^2-y$. In other words, $d_{x + y}+d_0\geq d_x+d_y$ for all $0\leq
x,y$ and $0\leq x+y<p^2$.  As this is symmetric in $x,y$ we may assume
$y\leq x$. Thus we are concerned with the condition
\begin{equation}\label{cond1}
d_{x + y}+d_0\geq d_x+d_y\mbox{ for all }0\leq y\leq x\leq
x+y<p^2.
\end{equation} 
We have the $p$-adic expressions: $x=x_{(0)}+x_{(1)}p$ and
$y=y_{(0)}+y_{(1)}p$.  When we add these expressions, we get
$x+y=c_{(0)}+c_{(1)}p+\epsilon_{(1)}p^2$ where $0\leq c_{(i)}\leq
p-1$, $x_{(0)}+y_{(0)}=c_{(0)}+p\epsilon_{(0)}$,
$\epsilon_{(0)}+x_{(1)}+y_{(1)}=c_{(1)}+p\epsilon_{(1)}$ and the
$\epsilon_{(i)}\in\{0,1\}$ depend upon whether there is a carry.  Note
$\epsilon_{(1)}=0$, since $x+y<p^2$.
Recall $b_2=b_1+p^2m$.  Replace $b_2$ in
(\ref{cond1}) with $b_2=b_1+p^2m$, and get
\begin{equation}\label{cond'}
\left\lfloor\frac{(1+x'+y' )b_1+\epsilon_{(0)}D}{p^2}\right \rfloor+
\left\lfloor\frac{b_1}{p^2}\right \rfloor \geq
\left\lfloor\frac{(1+x')b_1}{p^2}\right \rfloor+
\left\lfloor\frac{(1+y')b_1}{p^2}\right \rfloor,
\end{equation}
where $x'=x_{(1)}+px_{(0)}$, $y'=y_{(1)}+py_{(0)}$ and
$D=(b_2-p^2b_1)$, all over the same range of $x,y$.
 There are two cases to consider: $\epsilon_{(0)}=0$ and 
$\epsilon_{(0)}=1$. We consider the case
$\epsilon_{(0)}=1$ first.
Using $b_2=b_1+p^2m$, observe that (\ref{cyc-spread}) means $m\geq
b_1-\lfloor b_1/p^2\rfloor$ and thus by replacing $m$ in
$b_2=b_1+p^2m$ with $ b_1-\lfloor b_1/p^2\rfloor$, we find $D\geq
b_1-p^2\lfloor b_1/p^2\rfloor$. It is enough therefore to show that
(\ref{cond'}) with $\epsilon_{(0)}=1$ holds when $D$ is replaced by
$b_1-p^2\lfloor b_1/p^2\rfloor$.  In other words, it is enough to show
that
$$\left\lfloor\frac{(2+x'+y' )b_1}{p^2}\right \rfloor \geq
\left\lfloor\frac{(1+x')b_1}{p^2}\right \rfloor+
\left\lfloor\frac{(1+y')b_1}{p^2}\right \rfloor.$$ But this follows
from the generic fact: $\lfloor (a+b)/c\rfloor \geq \lfloor
a/c\rfloor+ \lfloor b/c\rfloor$ for positive integers $a,b,c$.  The
case of (\ref{cond'}) for those $x,y$ with $\epsilon_{(0)}=0$ (so that
$x_{(i)}+y_{(i)}<p$ for both $i=0,1$) is equivalent to
\cite[(6)]{byott:scaffold}, which, because of
\cite{byott:QJM,byott:scaffold}, is equivalent to $r(b)\mid p^2-1$.

\section{Examples: $p=2$ and $3$}
In this section, we determine necessary conditions for a Galois
scaffold to exist when $p=2, 3$.  Assuming the case $p=3$ to be
representative of the general case, $p$ odd, our results
suggest that the conditions in Theorem \ref{scaffold} are sharp.

We treat $p=2$ for the sake of completeness. Note that the condition
on the residue of the ramification numbers in Theorem \ref{assoc-main}
holds vacuously.  Consequently, {\em every} fully ramified $C_2\times
C_2$-extension possesses a Galois scaffold \cite[Thm
  5.1]{elder:scaffold}, and furthermore the ring of integers is free
over its associated order in {\em every} fully ramified $C_2\times
C_2$-extension \cite[Cor 1.3]{byott:scaffold}.  This suggests that
$p=2$ is a special case. It also explains why we only consider
$C_4$-extensions here.

\subsection{Outline}

Recall that a Galois scaffold for an extension of degree $p^2$ requires two
elements $\Psi_2,\Psi_1\in K_0[G]$ satisfying
(\ref{strong-scaff}), (\ref{scaff}). Here we 
outline a general procedure which, in principle, should enable us to
obtain a necessary condition for the existence of a Galois scaffold
for arbitrary $p$. In the remainder of this section, we implement this
procedure.

Adopt the notation of \S 2.1.  So whether $G\cong C_p\times
C_p$ or $C_{p^2}$, we have $K_1=K_0(x_1)$ with $v_1(x_1)=-b_1$. Our
first step is then to identify an element $X_2\in K_2$ such that
$v_2(X_2)=-b_2$. Once this is done, we have
\begin{equation}
\alpha_{i,j}=\binom{X_2}{i}\binom{x_1}{j}, \quad 0\leq i,j<p
\label{alpha}
\end{equation}
satisfying $v_2(\alpha_{i,j})=-ib_2-jpb_1$. So
$\{v_2(\alpha_{i,j}):0\leq i,j<p\}$ is a complete set of residues
modulo $p^2$, and thus $\mathcal{B}=\{\alpha_{i,j}:0\leq i,j<p\}$ is a
basis for $K_2$ over $K_0$.  Notice that $\alpha_{p-1,p-1}$ satisfies
$v_2(\alpha_{p-1,p-1})\equiv b_2\bmod p^2$

A basis for $K_0[G]$ is given by $\{(\sigma_2-1)^i(\sigma_1-1)^j:0\leq
i,j<p\}$.  Our next step in each case is to express
$(\sigma_2-1)^i(\sigma_1-1)^j\alpha_{p-1,p-1}$ in terms of
$\mathcal{B}$. The fact for each $0\leq i<p$ both
$(\sigma_2-1)^i(\sigma_1-1)^0\alpha_{p-1,p-1}$ and
$(\sigma_2-1)^{p-1}(\sigma_1-1)^i\alpha_{p-1,p-1}$ are expressed as a
single element of $\mathcal{B}$ motivates the use of binomial
coefficients to create our basis $\mathcal{B}$ (rather than the more
naive basis $\{X_2^ix_1^j:0\leq i,j<p\}$).

At this point, we are prepared to identify elements $\Theta_j\in K_0[G]$
for $0\leq j<p$ such that
$v_2(\Theta_j\alpha_{p-1,p-1})=v_2(\alpha_{p-1,p-1})+jpb_1$. They
exist because $\alpha_{p-1,p-1}$ generates a normal basis
\cite{elder:cor-criterion}.  Because
$\{v_2((\sigma_2-1)^i\Theta_j\alpha_{p-1,p-1}):0\leq i,j<p\}$ is a
complete set of residues, $K_2= \sum_{0\leq i,j<p}K_0\cdot
(\sigma_2-1)^i\Theta_j\alpha_{p-1,p-1}$. Therefore
$\{(\sigma_2-1)^i\Theta_j:0\leq i,j<p\}$ is a basis for $K_0[G]$.  

If there is a Galois scaffold there must be $\Psi_2, \Psi_1$ in the
augmentation ideal $(\sigma-1:\sigma\in G)$ of $K_0[G]$ satisfying
\eqref{strong-scaff} and \eqref{scaff}.  Because of \eqref{scaff},
there exist $0\leq i,j<p$ such that
$v_2(\Psi_2^i\Psi_1^j\alpha_{p-1,p-1})\equiv
v_2(\alpha_{p-1,p-1})+pb_1\bmod p^2$.  Thus
$v_2(a\Psi_2^i\Psi_1^j\alpha_{p-1,p-1})= v_2(\alpha_{p-1,p-1})+pb_1$
for some $a\in K_0$.  Clearly $a\Psi_2^i\Psi_1^j\in(\sigma-1:\sigma\in
G)^{i+j}$.  Lemma 3.1 below gives $i+j=1$. Thus, without loss of generality,
we assume $i=0$ and $j=1$ and that $v_2(\Psi_1\alpha_{p-1,p-1})=
v_2(\alpha_{p-1,p-1})+pb_1$.  Note that the augmentation ideal
$(\sigma-1:\sigma\in G)$ of $K_0[G]$ is also its Jacobson radical and
unique maximal ideal.  Express $\Psi_1=\sum_{0\leq
  i,j<p}a_{i,j}(\sigma_2-1)^i\Theta_j$ for some $a_{i,j}\in K_0$ with
$a_{0,0}=0$, and proceed to impose the first requirement of a Galois
scaffold, namely (\ref{strong-scaff}). How? This depends upon $p$.

\begin{lemma}
Given $\alpha\in K_2$ with $v_2(\alpha)\equiv b_2\bmod p^2$. If
$\theta$ lies in the augmentation ideal of $K_0[G]$,
$(\sigma-1:\sigma\in G)$, and $v_2(\theta\alpha)=v_2(\alpha)+pb_1$,
then $\theta\not\in (\sigma-1:\sigma\in G)^2$.
\end{lemma}
\begin{proof}
Let $\mbox{Tr}_{K_i/K_j}=(\sigma_2-1)^{p-1}$ denote the trace from $K_i$ down to $K_j$.
Using \cite[V\S3 Lemma 4]{serre:local},
$v_1(\mbox{Tr}_{K_2/K_1}\alpha)=(v_2(\alpha)+(p-1)b_2)/p \equiv
b_2\equiv b_1\bmod p$. So
$v_1((\sigma_1-1)^i\mbox{Tr}_{K_i/K_j}\alpha)\equiv (i+1)b_1\bmod p$
for $0\leq i<p$. 
It is also the case that
$v_1(\mbox{Tr}_{K_2/K_1}\theta\alpha)=(v_2(\alpha)+pb_1+(p-1)b_2)/p
\equiv b_2+b_1\equiv 2b_1\bmod p$.
In particular, 
$v_1(\mbox{Tr}_{K_2/K_1}\theta\alpha)<\infty$.
Let $\theta=\sum_{0\leq i,j<p}a_{i,j}(\sigma_1-1)^i(\sigma_2-1)^j$ with 
$a_{i,j}\in K_0$.  Since $\theta$ lies in the 
augmentation ideal of $K_0[G]$, $a_{00}=0$.
If $\theta\in
(\sigma-1:\sigma\in G)^2$, then  $a_{10}=0$ as well.
As a result, $\mbox{Tr}_{K_2/K_1}\theta\alpha=
\sum_{i=2}^{p-1}a_{i,0}(\sigma_1-1)^i(\sigma_2-1)^{p-1}\alpha=\sum_{i=2}^{p-1}a_{i,0}(\sigma_1-1)^i\mbox{Tr}_{K_2/K_1}\alpha$.
If $p=2$, the contradiction arises 
because we can not have both $v_1(\mbox{Tr}_{K_2/K_1}\theta\alpha)<\infty$ and
$\mbox{Tr}_{K_2/K_1}\theta\alpha=0$. If $p>2$, the contraction arises because 
for $2\leq i<p$, $2b_1\not\equiv (i+1)b_1\bmod p$.
\end{proof}

\subsection{$C_4$-extensions}

There are two conditions stated in Theorem \ref{scaffold}. They are sufficient
for a Galois scaffold. For $p=2$, one of these conditions holds
vacuously, which leaves $b_2>4b_1$, namely (\ref{cyc-spread}), as the
only interesting condition. Here we show that 
$b_2\geq
4b_1-1$ is both necessary and sufficient for a Galois scaffold to
exist in a fully ramified $C_4$-extension.
Assume notation of \S2.1.2.  So $v_0(\beta_1)=-b_1<0$ odd, and
$K_2/K_0$ satisfies $K_2=K_0(x_2)$ with
$\wp(x_2)=\beta_1x_1+\mu^2\beta_1+\epsilon$ where $\mu\in K_0$, and
because $p=2$, $\epsilon\in \kappa$.  Recall
$(\sigma_1-1)x_2=x_1$ with $\wp(x_1)=\beta_1$.  Let $m=-v_0(\mu)$ and
$X_2=x_2-\mu x_1\in K_2$. Then
$\wp(X_2)=(\beta_1+\wp(\mu))x_1+\epsilon$ where
$v_1((\beta_1+\wp(\mu))x_1)=-\max\{3b_1,b_1+4m\}=-b_2$.  Thus
$v_2(X_2)=-b_2$. The basis $\mathcal{B}$ is $\{\alpha_{i,j}:0\leq
i,j\leq 1\} =\{1, x_1, X_2, X_2x_1\}$.

Note that $(\sigma_1-1)X_2=x_1-\mu$. So $\sigma_1X_2x_1=(X_2+x_1-\mu)(x_1+1)$
and thus $(\sigma_1-1)X_2x_1=X_2+\mu x_1+\beta_1+\mu$. Therefore
for $0\leq i,j\leq 1$ we have
$(\sigma_2-1)^i(\sigma_1-1)^j\alpha_{1,1}=\alpha_{1-i,1-j}+\epsilon_{i,j}$,
where the error term $\epsilon_{i,j}$ is zero for $(i,j)\in
\{(0,0),(1,0),(1,1)\}$, and
$\epsilon_{0,1}=\mu\alpha_{0,1}+(\beta_1+\mu)\alpha_{0,0}$.
Use this to find
$\Theta_1=(\sigma_1\sigma_2^{[\mu]}-1)+\beta_1(\sigma_1-1)(\sigma_2-1)$,
so that 
the effect of $\Theta_1$, $(\sigma_2-1)$, $(\sigma_2-1)\Theta_1$
on $\alpha_{1,1}=X_2x_1$, $\alpha_{1,0}=X_2$, $\alpha_{0,1}=x_1$ is as follows:
\begin{equation}\label{table}
\begin{array}{r|ccc}
&X_2x_1&X_2&x_1\\ \hline
\Theta_1& X_2 &x_1&1\\
(\sigma_2-1)&x_1 &1&0\\
(\sigma_2-1)\Theta_1& 1&0&0
\end{array}\end{equation}

Now $\alpha_{1,1}$ satisfies the valuation criterion for a normal
basis generator, namely $v_2(\alpha_{1,1})=b_2 \mod 4$. Thus, if there is a
Galois scaffold, then there is a $\Psi_1$ in the augmentation ideal of
$K_0[G]$, which is expressible as
$\Psi_1=a_{0,1}\Theta_1+a_{1,0}(\sigma_2-1)+a_{1,1}(\sigma_2-1)\Theta_1$
with $a_{i,j}\in K_0$, such that $v_2(\Psi_1 \rho )=v_2(\rho)+2b_1$
for all $\rho\in K_2$ with $v_2(\rho)\equiv v_2(\alpha_{1,1})\bmod 4$.
Since $v_2(\Theta_1 \alpha_{1,1})=v_2(\alpha_{1,1})+2b_1$,
$v_2(a_{0,1})=0$ and $v_2(a_{i,j})+ib_2+2jb_1>2b_1$ for $(i,j)\neq
(0,1)$.  Multiplying $\rho$ by an element of $K_0$ if necessary, we
may assume, without loss of generality, that $\rho=X_2x_1
+aX_2+bx_1+c$ with $a,b,c\in K_0$.  So that
$v_2(\rho)=v_2(\alpha_{1,1})$, we require $v_2(a)>-2b_1$,
$v_2(b)>-b_2$ and $v_2(c)>-b_2-2b_1$.  Note that
$\Psi_1\rho=a_{0,1}X_2+(a_{0,1}a+a_{1,0})x_1
+(a_{0,1}b+a_{1,0}a+a_{1,1})$.  Using the bounds on $v_2(a)$,
$v_2(b)$, $v_2(c)$ and the $v_2(a_{i,j})$, we find $\Psi_1\rho\equiv
a_{0,1}(X_2+ax_1) \bmod X_2\euP_2$, which means that for a Galois
scaffold we require $v_2(ax_1)>v_2(X_2)$ for all $a\in K_0$ with
$v_0(a)\geq \lceil -2b_1/4\rceil$.  Thus $\lfloor
b_1/2\rfloor\leq\lfloor(b_2-2b_1)/4\rfloor$, which since $b_1$ is odd
is equivalent to $b_2\geq 4b_1-1$.  On the other hand, if $b_2\geq
4b_1-1$ a Galois scaffold exists. This follows from the observation
that for $\rho^*\in K_2$ with $v_2(\rho^*)$ odd, we have
$v_2((\sigma_2-1)\rho^*)=v_2(\rho^*)+b_2$.

\subsection{$C_9$-extensions}
We prove that for $C_9$-extensions the conditions in Theorem
\ref{scaffold} are sharp.
Assume $p=3$ in \S2.1.2.  So
$v_0(\beta_1)=-b_1<0$ with $p\nmid b_1$. Either $v_0(\beta_2)<0$ with
$p\nmid v_0(\beta_2)$ or $\beta_2\in \kappa$. 
In any case, there are $\mu_1$, $\mu_2 \in K_0$  (either or 
both of which may be $0$) 
and $k\in \kappa$
such that 
$\beta_2=\mu_1^3\beta_1+\mu_2^3\binom{\beta_1}{2}+k$.
Let 
$m_i=-v_0(\mu_i)$ for $i=1$, $2$. If $\beta_2 \neq 0$,
$v_0(\beta_2)=-\max\{3m_1+b_1,3m_2+2b_1,0\}$.  Our $C_9$-extension
$K_2/K_0$ satisfies $K_2=K_0(x_2)$ with
$\wp(x_2)=-\beta_1x_1^2-\beta_1^2x_1+\mu_1^3\beta_1+\mu_2^3\binom{\beta_1}{2}+k$,
and $(\sigma_1-1)x_2=-x_1^2-x_1$ with $\wp(x_1)=\beta_1$.  Let
$X_2=x_2-\mu_1 x_1-\mu_2\binom{x_1}{2}$. Then $\wp(X_2)
=-\beta_1x_1^2-\beta_1^2x_1-\wp(\mu_1)x_1-\wp(\mu_2)\binom{x_1}{2}-\mu_2^3\beta_1x_1+k$. Notice
that $v_1(\wp(X_2))=-b_2$ and
$(\sigma_1-1)X_2=-x_1^2-x_1-\mu_1-\mu_2x_1$.  Thus $v_2(X_2)=-b_2$. We
have our basis $\mathcal{B}=\left \{\alpha_{i,j}:0\leq i,j\leq 2\right
\}$ using (\ref{alpha}).

Verify, using a software package like Maple, that
for $0\leq i,j\leq 2$ we have
$(\sigma_2-1)^i(\sigma_1-1)^j\alpha_{2,2}=\alpha_{2-i,2-j}+\epsilon_{i,j}$,
where the error term $\epsilon_{i,j}$ is zero for $(i,j)\in
\{(0,0),(1,0),(2,0).(2,1),(2,2)\}$. Otherwise
\begin{eqnarray*}
\epsilon_{0,1}&=&(1-\mu_1-\mu_2)(\alpha_{1,2}+\alpha_{1,1})+\beta_1\alpha_{1,1}+(\mu_2-1)\beta_1\alpha_{1,0}
\\ & &
+(\mu_1\mu_2+\mu_1-\mu_1^2+\mu_2-\mu_2^2)(\alpha_{0,2}+\alpha_{0,1})+\mu_2\beta_1\alpha_{0,2}\\& & +(\mu_2^2-\mu_1)\beta_1\alpha_{0,1}+
((\mu_1-\mu_2-\mu_1\mu_2+\mu_2^2)\beta_1 +\beta_1^2)\alpha_{0,0}\\ \epsilon_{1,1}
&=&(1-\mu_1-\mu_2)(\alpha_{0,2}+\alpha_{0,1})+\beta_1\alpha_{0,1}+(\mu_2-1)\beta_1\alpha_{0,0}
\end{eqnarray*}
\begin{eqnarray*}
\epsilon_{0,2}&=& (\mu_2-1)(\alpha_{1,2}-\alpha_{1,0})+\mu_1(\alpha_{1,1}+\alpha_{1,0}) \\
&&+(\mu_2^2-\mu_2+\mu_1^2-(1+\mu_2)\beta_1)\alpha_{0,2}\\
& & 
+(-\mu_1\mu_2-\mu_1+(1+\mu_2-\mu_2^2+\mu_1)\beta_1)\alpha_{0,1}\\
& & +(\mu_2-\mu_2^2-\mu_1-\mu_1^2-\mu_1\mu_2+(\mu_1\mu_2-\mu_1-\mu_2-\mu_2^2-1)\beta_1-\beta_1^2)\alpha_{0,0}
\\
\epsilon_{1,2}&=& (\mu_2-1)(\alpha_{0,2}-\alpha_{0,0})+\mu_1(\alpha_{0,1}+\alpha_{0,0})
 \end{eqnarray*}
Now observe that because 
$b_2= \max\{7b_1,b_1+9m_1,4b_1+9m_2\}$, we have $v_2(\mu_1)\geq
b_1-b_2$ and $v_2(\mu_2)\geq 4b_1-b_2$. So
using the expressions for $(\sigma_2-1)^i(\sigma_1-1)^j\alpha_{2,2}$
and
Gaussian Elimination,
we find that
\begin{eqnarray*}
\Theta_1&=&
(\sigma_1\sigma_2^{[\mu_1]}-1)-\beta_1(\sigma_2-1)(\sigma_1-1)
-\mu_2\beta_1(\sigma_2-1)(\sigma_1-1)^2\\
&&+(\mu_2^2-\mu_1)\beta_1(\sigma_2-1)^2
+\left [(\mu_1\mu_2-\mu_1-\mu_2^2)\beta_1+\beta_1^2\right ](\sigma_2-1)^2(\sigma_1-1)\\
&&+\left [\mu_1\mu_2\beta_1+(1+\mu_2)\beta_1^2\right ](\sigma_2-1)^2(\sigma_1-1)^2,\\
\Theta_2&=&(\sigma_1\sigma_2^{[\mu_1]}-1)^2
-\mu_2(\sigma_2-1)
+(1+\mu_2)\beta_1(\sigma_2-1)^2\\
&&+(\mu_2^2-\mu_2)\beta_1(\sigma_2-1)^2(\sigma_1-1)
+[(\mu_2+\mu_2^2)\beta_1+\beta_1^2](\sigma_2-1)^2(\sigma_1-1)^2,
\end{eqnarray*}
give $\Theta_j\alpha_{2,2}\equiv\alpha_{2,2-j}\bmod \alpha_{2,2-j}\euP_2$. Let $\Theta_0=1$.  

If there is a Galois scaffold then there is a $\Psi_1$ in the augmentation ideal of $K_0[G]$,
which is expressible as $\Psi_1=\sum_{0\leq i,j\leq
  2}a_{i,j}(\sigma_2-1)^i\Theta_j$ for some $a_{i,j}\in K_0$, such
that $v_2(\Psi_1 \alpha_{2,2} )=v_2(\alpha_{2,2})+3b_1$ and
$v_2(\Psi_1^2\alpha_{2,2})=v_2(\alpha_{2,2})+6b_1$.  
Note that $a_{0,0}=0$, since
$\Psi_1$ is in the augmentation ideal. Since
$v_2(\Psi_1 \alpha_{2,2} )=v_2(\alpha_{2,2})+3b_1$, we have $v_2(a_{0,1})=0$
and $v_2(a_{i,j})+ib_2+3jb_1>3b_1$ for $(i,j)\neq (0,1)$.

To determine $v_2(\Psi_1^2\alpha_{2,2})$, we expand $\Psi_1^2$ in
terms of the $K_0$-basis $\{(\sigma_2-1)^i\Theta_j:0\leq i,j\leq 2\}$
for $K_0[G]$. This requires the following identities, which can
be verified with a software package like Maple (establish
polynomial identities where $x=\sigma_1-1$, $x^3=\sigma_2-1$,
and $x^9=0$):
\begin{eqnarray*}
\Theta_1^2&=&\Theta_2+ \beta_1(\sigma_2-1)\Theta_2
-(\mu_2\beta_1+\beta_1^2)(\sigma_2-1)^2\Theta_2
\\
& & +\beta_1(\mu_2^2+\mu_2)(\sigma_2-1)^2\Theta_1+(\mu_2-1)\beta_1(\sigma_2-1)^2+\mu_2(\sigma_2-1),\\
\Theta_1\Theta_2&=&(\sigma_2-1)-\mu_2(\sigma_2-1)\Theta_1+\beta_1(\sigma_2-1)^2\Theta_1-\beta_1(\mu_2+\mu_2^2)(\sigma_2-1)^2\Theta_2\\
& & -\beta_1(\sigma_2-1)^2,\\
\Theta_2^2&=&(\sigma_2-1)\Theta_1+\beta_1(\sigma_2-1)^2\Theta_1
+\mu_2(\sigma_2-1)\Theta_2
-\beta_1(\sigma_2-1)^2\Theta_2\\
&&-\mu_2^2(\sigma_2-1)^2.
\end{eqnarray*}

In the expansion of $\Psi_1^2$ in terms of
$\{(\sigma_2-1)^i\Theta_j:0\leq i,j\leq 2\}$, we find the coefficient
of $(\sigma_2-1)$ to be $2a_{0,1}a_{0,2}+
a_{0,1}^2\mu_2$, while the coefficient of $(\sigma_2-1)\Theta_2$ is
$2a_{0,1}a_{1,1}+2a_{1,0}a_{0,2}
+a_{0,2}^2\mu_2+a_{0,1}^2\beta_1$.  When we apply $\Psi_1^2$ to
$\alpha_{2,2}$, it must be that both
$v_2((2a_{0,1}a_{0,2}+a_{0,1}^2\mu_2)(\sigma_2-1)\alpha_{2,2})>v_2(\Theta_2\alpha_{2,2})=v_2(\alpha_{2,2})+6b_1$
and
$v_2((2a_{0,1}a_{1,1}+2a_{1,0}a_{0,2}+a_{0,2}^2\mu_2+a_{0,1}^2\beta_1)(\sigma_2-1)\Theta_2\alpha_{2,2})>v_2(\alpha_{2,2})+6b_1$.
  We may discard those terms of valuation greater than
  $v_2(\alpha_{2,2})+6b_1$, using $v_2(a_{0,1})=0$ and
  $v_2(a_{i,j})+ib_2+3jb_1>3b_1$ for $(i,j)\neq (0,1)$.  This means
  that we can drop $-a_{0,1}a_{1,1}-a_{1,0}a_{0,2}$ from the
  coefficient for $(\sigma_2-1)\Theta_2\alpha_{2,2}$, leaving
  $(a_{0,2}^2\mu_2+a_{0,1}^2\beta_1)
  (\sigma_2-1)\Theta_2\alpha_{2,2}$.  If $v_2(\mu_2)<6b_1-b_2$, then
  because
  $v_2(a_{0,1}^2\mu_2(\sigma_2-1)\alpha_{2,2})<v_2(\alpha_{2,2})+6b_1$
  we must have $v_0(\mu_2)=v_0(a_{0,2})$.  If $b_2<9b_1$, then because
  $v_2(a_{0,1}^2\beta_1(\sigma_2-1)\Theta_2\alpha_{2,2})<v_2(\alpha_{2,2})+6b_1$
  we must have $2v_0(a_{0,2})+v_0(\mu_2)=v_0(\beta_1)$. So if
  $v_2(\mu_2)<6b_1-b_2$ and $b_2<9b_1$, then we must have
  $3v_0(\mu_2)=v_0(\beta_1)=-b_1$. But $3\nmid b_1$.  This means that
  we have $v_2(\mu_2)>6b_1-b_2$ or $b_2>9b_1$, and there are two cases
  to consider.  Suppose that $v_2(\mu_2)>6b_1-b_2$. Then because we
  must have $v_2(
  (a_{0,2}^2\mu_2+a_{0,1}^2\beta_1)(\sigma_2-1)\Theta_2\alpha_{2,2}
  )>v_2(\alpha_{2,2})+6b_1$, we must have $v_2(
  (a_{0,1}^2\beta_1)(\sigma_2-1)\Theta_2\alpha_{2,2}
  )>v_2(\alpha_{2,2})+6b_1$, or $b_2>9b_1$.  Suppose that $b_2>9b_1$.
  Then because we must have
  $v_2((a_{0,1}^2\mu_2-a_{0,1}a_{0,2})(\sigma_2-1)\alpha_{2,2})>v_2(\alpha_{2,2})+6b_1$,
  we must have
  $v_2(a_{0,1}^2\mu_2(\sigma_2-1)\alpha_{2,2})>v_2(\alpha_{2,2})+6b_1$,
  or $v_2(\mu_2)>6b_1-b_2$.  As a result, we have shown that in order
  for a Galois scaffold to exist, both $b_2>9b_1$ and
  $v_2(\mu_2)>6b_1-b_2$ must hold.  The first condition agrees with
  (\ref{cyc-spread}). The second condition agrees with
  (\ref{cyc-shape}).

\subsection{$C_3\times C_3$-extensions}
We prove that for $C_3\times C_3$-extensions the conditions in Theorem
\ref{scaffold} are sharp.
Assume $p=3$ in \S2.1.1.  So $v_0(\beta_2)\leq v_0(\beta_1)=-b_1<0$
with $p\nmid b_1, v_0(\beta_2)$.  We follow \S3.3 closely, except that
the technical issues here are easier, since 
expressions here are often truncations of the expressions in
\S3.3: Again, there are elements $\mu_1,\mu_2\in K_0$ with
$v_0(\mu_i)=-m_i$ and a $k\in \kappa$ such that
$\beta_2=\mu_1^3\beta_1+\mu_2^3\binom{\beta_1}{2}+k$. Since
$v_0(\beta_2)<0$, $v_0(\beta_2)=-\max\{3m_1+b_1,3m_2+2b_1\}$.  Let
$\wp(x_i)=\beta_i$, and let $(\sigma_i-1)x_j=\delta_{i,j}$ be the
Kronecker delta function.  Let $X_2=x_2-\mu_1
x_1-\mu_2\binom{x_1}{2}$. Then $\wp(X_2)
=-\wp(\mu_1)x_1-\wp(\mu_2)\binom{x_1}{2}-\mu_2^3\beta_1x_1+k$. So
$v_1(\wp(X_2))=-b_2$ and $(\sigma_1-1)X_2=-\mu_1-\mu_2x_1$.  A basis
for $K_2/K_0$ is given by $\mathcal{B}=\left \{\alpha_{i,j}:0\leq
i,j\leq 2\right \}$ where $\alpha_{i,j}=\binom{X_2}{i}\binom{x_1}{j}$.

Verify, using a software package as in \S3.3, that
for $0\leq i,j\leq 2$ we have
$(\sigma_2-1)^i(\sigma_1-1)^j\alpha_{2,2}=\alpha_{2-i,2-j}+\epsilon_{i,j}$,
where the error term $\epsilon_{i,j}$ is zero for $(i,j)\in
\{(0,0),(1,0),(2,0).(2,1),(2,2)\}$. Otherwise
\begin{eqnarray*}
\epsilon_{0,1} &=&
-(\mu_1+\mu_2)(\alpha_{1,2}+\alpha_{1,1})+\mu_2\beta_1\alpha_{1,0}\\
&&+(\mu_1\mu_2-\mu_1-\mu_1^2-\mu_2-\mu_2^2)(\alpha_{0,2}
+\alpha_{0,1})+
\mu_2^2\beta_1\alpha_{0,1} \\
&&+(\mu_2-\mu_1\mu_2+\mu_2^2)\beta_1\alpha_{0,0},\\
\epsilon_{1,1} &=&
-(\mu_1+\mu_2)(\alpha_{0,2}+\alpha_{0,1})+\mu_2\beta_1\alpha_{0,0},\\
\epsilon_{0,2} &=&
\mu_2(\alpha_{1,2}-\alpha_{1,0})+\mu_1(\alpha_{1,1}+\alpha_{1,0})
+(\mu_2^2+\mu_2+\mu_1^2)\alpha_{0,2} \\
&&+(\mu_1-\mu_1\mu_2
-\mu_2^2\beta_1)\alpha_{0,1}\\
&&+(\mu_1-\mu_1^2-\mu_2-\mu_2^2-\mu_1\mu_2
+(\mu_1\mu_2
-\mu_2^2)\beta_1)\alpha_{0,0},\\
\epsilon_{1,2}
&=&\mu_2(\alpha_{0,2}-\alpha_{0,0})+\mu_1(\alpha_{0,1}+\alpha_{0,0}).\\
\end{eqnarray*}
Use this and the fact that because
$b_2= \max\{b_1+9m_1,4b_1+9m_2\}$, we have $v_2(\mu_1)\geq
b_1-b_2$ and $v_2(\mu_2)\geq 4b_1-b_2$ to find
that 
$\Theta_j\alpha_{2,2}\equiv\alpha_{2,2-j}\bmod \alpha_{2,2-j}\euP_2$ for
\begin{multline*}
\Theta_1=(\sigma_1\sigma_2^{[\mu_1]}-1) +\mu_2^2\beta_1(\sigma_2-1)^2
-\mu_2\beta_1(\sigma_2-1)(\sigma_1-1)^2\\+(\mu_1\mu_2-\mu_2^2)\beta_1(\sigma_2-1)^2(\sigma_1-1)
+(\mu_1\mu_2-\mu_2)\beta_1(\sigma_2-1)^2(\sigma_1-1)^2,
\end{multline*}
\begin{multline*}
\Theta_2=
(\sigma_1\sigma_2^{[\mu_1]}-1)^2 -\mu_2(\sigma_2-1)
-\mu_2(\sigma_2-1)^2
+\mu_2^2\beta_1(\sigma_2-1)^2(\sigma_1-1)\\
+\mu_2^2\beta_1(\sigma_2-1)^2(\sigma_1-1)^2.
\end{multline*}
Let $\Theta_0=1$.
Using a software package as in \S3.3, we establish:
\begin{eqnarray*}
\Theta_1^2&=&\Theta_2 +\mu_2(\sigma_2-1)+\mu_2(\sigma_2-1)^2+\mu_2^2\beta_1(\sigma_2-1)^2\Theta_1,\\
\Theta_1\Theta_2&=&-\mu_2(\sigma_2-1)\Theta_1-\mu_2(\sigma_2-1)^2\Theta_1-\mu_2^2\beta_1(\sigma_2-1)^2\Theta_2\\
\Theta_2^2&=&
-\mu_2^2(\sigma_2-1)^2
+\mu_2(\sigma_2-1)\Theta_2
+\mu_2(\sigma_2-1)^2\Theta_2
\end{eqnarray*}

If there is a Galois scaffold then there is a $\Psi_1=\sum_{0\leq
  i,j\leq 2}a_{i,j}(\sigma_2-1)^i\Theta_j$ in the augmentation ideal of $K_0[G]$
with
$a_{i,j}\in K_0$ and $a_{0,0}=0$,
such that $v_2(\Psi_1\alpha_{2,2})=v_2(\alpha_{2,2})+3b_1$ and thus
$v_2(a_{0,1})=0$ and for $(i,j)\neq (0,1)$,
$v_2(a_{i,j})+ib_2+3jb_1>3b_1$. 
Furthermore
$v_2(\Psi_1^2\alpha_{2,2})=v_2(\alpha_{2,2})+6b_1$.  Expand
$\Psi_1^2$ in terms of $\{(\sigma_2-1)^i\Theta_j:0\leq i,j\leq
2\}$. The coefficient of $(\sigma_2-1)$ is
$a_{0,1}^2\mu_2$.  When we apply $\Psi_1^2$ to
$\alpha_{2,2}$, we must have
$v_2(a_{0,1}^2\mu_2(\sigma_2-1)\alpha_{2,2})>v_2(\Theta_2\alpha_{2,2})=v_2(\alpha_{2,2})+6b_1$.
This implies
$v_2(\mu_2)>6b_1-b_2$ and thus (\ref{cyc-shape}).

\bibliography{bib}
\end{document}